\documentclass[11pt]{article}
\usepackage{amssymb,amsmath,latexsym,ae,aecompl}
\usepackage{graphics}
\usepackage{latexsym}
\usepackage{algorithm}
\usepackage{algorithmic}
\usepackage{tikz}
\newcommand{\qed}{\hfill$\Box$}

\newenvironment{proof}{\noindent {\bf Proof.}}{\qed}

\newtheorem{theorem}{Theorem}[section]

\begin{document}

\baselineskip 0.2in
\parskip      0.1in
\parindent    0em

\title{{\bf On the equivalence of the Mizoguchi-Takahashi locally contractive map to Nadler's locally contractive map} \footnotemark[3]}
\author{
Asrifa Sultana \footnotemark[1]
\and
Xiaolong Qin \footnotemark[2]
}

\date{ }
\maketitle
\def\thefootnote{\fnsymbol{footnote}}

\footnotetext[1]{
\noindent
Department of Mathematics, Indian Institute of Technology Bhilai, Chhattisgarh 492015, India}

\footnotetext[2]{
\noindent
Institute for Fundamental and Frontier Science, University of Electronic Science and Technology of China, China }

\footnotetext[3]{This article is an initial draft of the complete paper}
\begin{abstract}
In this article we have proved the equivalence between the Mizoguchi-Takahashi uniformly~locally~contractive map to the multi-valued map satisfying the Nadler contractive condition uniformly~locally~on a metrically convex space.
\end{abstract}
\pagebreak

\section{Introduction}

In 1969, Nadler \cite{N6} established a multi-valued extension of the famous Banach contraction principle. Let $(X,d)$ be a complete metric space and
$CB(X)$ denote the family of all non-empty closed and bounded subsets of $X$. Let $H$ be the Hausdorff metric with respect to $d$ on $CB(X)$. Nadler \cite{N6} proved that any multi-valued map $F$ from $X$ to $CB(X)$ has a fixed point if there exists some $k \in [0,1)$ such that $H(F(x),F(y))\leq kd(x,y)$ for all $x,y \in X$. An interesting extension of this result was obtained in 1989, by Mizoguchi and Takahashi \cite{MT2}.

\begin{theorem}\cite{MT2} \label{Th 2 Mizoguchi}
Let $(X,d)$ be a complete metric space and let $F:X \rightarrow CB(X)$ be a map such that
\begin{equation}
H(F(x),F(y))\leq \alpha(d(x,y))d(x,y)~\forall~x,y \in X,~x \neq y,\label{eq 1 MT}
\end{equation}
where $\alpha:(0,\infty)\rightarrow [0,1)$ is such that $\limsup_{s \rightarrow t+}\alpha(s)<1$ for every $t\in[0,\infty)$. Then $F$ has a fixed point.
\end{theorem}
It is worth to note that any Mizoguchi-Takahashi contraction on a metric space need not be a Nadler contraction (for example, one can refer \cite{Suzuki,Antony_Veeramani}).

On the other hand, Edelstein \cite{E5} introduced the notion of uniformly locally contractive mapping on a metric space. For a metric space $(X,d)$, a mapping $f:X \rightarrow X$ is called an $(\varepsilon,k)$-\textit{uniformly~locally~contractive} (where $\varepsilon>0$ and $k\in[0,1)$) if $d(fx,fy)\leq kd(x,y)$ for all $x,y \in X$ with $d(x,y)<\varepsilon$. It is worth to note that an $(\varepsilon,k)$-uniformly~locally~contractive need not be a Banach contraction (for example, one can refer \cite{E5,mate}). Edelstein \cite{E5} established the following extension of Banach contraction principle.
  \begin{theorem}\cite{E5} \label{Th edelstein}
 Let $(X,d)$ be a complete metric space. An $(\varepsilon,k)$-\textit{uniformly~locally~contractive} map $f:X \rightarrow X$ has a unique fixed point if $(X,d)$ is $\varepsilon$-\textit{chainable}, that is, for any given $a,b \in X$, there exist $N \in \mathbb{N}$ and a sequence $(y^{i})_{i=0}^{N}$ in $X$ such that $y^{0}=a$, $y^{N}=b$ and $d(y^{i-1},y^{i})<\varepsilon$ for $1\leq i \leq N$.
  \end{theorem}

 In \cite{N6}, Nadler extended this result to multi-valued mappings. Let $(X,d)$ be a metric space and  $H$ be the Hausdorff metric with respect to $d$ on the family $CB(X)$ of all non empty closed and bounded subsets of $X$. Nadler \cite{N6} generalized the above result by deriving the following theorem.
\begin{theorem}\cite{N6}
Let $(X,d)$ be a complete metric space and $F:X \rightarrow CB(X)$ be a mapping such that for all~$x,y \in X$~with $d(x,y)<\varepsilon$,
\begin{equation}
H(F(x),F(y))\leq kd(x,y)\qquad~\textrm{for~some}~k \in [0,1).
\end{equation}
Then the map $F$ has a fixed point if $(X,d)$ is $\varepsilon$-chainable.
\end{theorem}

Recently, Eldred et al. \cite{Antony_Veeramani} explored some spaces on which the Mizoguchi-Takahashi contraction is equivalent to Nadler contraction. They have proved that Mizoguchi-Takahashi's condition reduces to a multi-valued contraction in a metrically convex space. Also, they have derived the equivalence in a compact metric space.

In this paper, we have shown that the multi-valued map satisfying the Mizoguchi-Takahashi's contractive condition uniformly locally~ is equivalent to a uniformly locally~contractive multi-valued map due to Nadler \cite{N6} on a metrically convex space.

\section{Preliminaries}
In this section we give some definitions and notations which are useful and related to context of our results.

Let $(X,d)$ be a metric space and $p,q$ be any two arbitrary points in $X$. A point $r\in X$ is called \textit{metrically} between $p$ and $q$ if $d(p,q)=d(p,r)+d(r,q)$ where $p \neq q \neq r$. We say the metric space $(X,d)$ \textit{metrically convex} if for any two arbitrary points $p$ and $q$, there is a point $r$ in $X$ which is metrically between $p$ and $q$.

A subset $M$ of a metric space $(X,d)$ is said to be a \textit{metric segment} with joining points $p,q \in M$ if there is a closed and bounded interval $[a,b] \subseteq \mathbb{R}$ and an isometry $\phi$ from the set $[a,b]$ onto the set $M$ such that $\phi(a)=p$ and $\phi(b)=q$.

Now, we recall the following result due to Khamsi and Kirk \cite{Khamsi_Kirk} which will be used in our main result.
\begin{theorem}
Let $(X,d)$ be complete metrically convex metric space. Then any two arbitrary points in $X$ are the joining points of at least one metric segment in $X$.
\end{theorem}

For the given metric space $(X,d)$, the notation $CB(X)$ denotes the family of all non empty closed and bounded subsets of $X$. For $A,B \in CB(X)$, let
$$
H(A,B)=\max\left\{\sup_{a \in A}d(a,B), \sup_{b \in B}d(b,A)\right\},
$$
where $d(a,B)=inf_{b \in B}d(a,b)$.
Then the map $H$ is a metric on $CB(X)$ which is called the $Hausdorff~metric$ induced by $d$.
%
%
%

\section{Main results}
The following result is the main result proved by us in this article.
\begin{theorem} \label{Th 1}
Let $(X,d)$ be a metrically convex complete metric space and $F:X\rightarrow CB(X)$ be such that $\forall x,y \in X$ with $d(x,y)<\varepsilon$ (where $\varepsilon>0$),
\begin{equation}
H(F(x),F(y))\leq \alpha(d(x,y))d(x,y),\label{eq 1 MT}
\end{equation}
where $\alpha:[0,\infty)\rightarrow [0,1)$ is such that $\limsup_{s \rightarrow t+}\alpha(s)<1$ for every $t\in[0,\infty)$.
Then $F$ is an $(\varepsilon,k)$-\textit{uniformly~locally~contractive~multi-valued mapping} (for some $k\in[0,1)$).
\end{theorem}
\begin{proof}
Our aim here is to prove that $F$ is an $(\varepsilon,k)$-uniformly~locally~contractive~multi-valued mapping, that is,
there exists $k \in [0,1)$ such that $H(F(x),F(y))\leq kd(x,y)~\forall x,y \in X~\textrm{with}~d(x,y)<\varepsilon$. Let us consider a
subset $P$ of real numbers where
$$P=\{p>0: \sup\{\alpha(d(x,y)): 0 \leq d(x,y) \leq p\}=1\}.$$
Then the following two cases occurs. Our aim is to show that $F$ becomes an $(\varepsilon,k)$-uniformly~locally~contractive~mapping in
each case. Let $x_{1}$, $x_{2}$ be two arbitrary elements in $X$ such that $d(x_{1},x_{2})<\varepsilon$.\vspace{0.2 cm}\\
\textbf{Case 1:} $P=\emptyset$.\\
Therefore $\sup\{\alpha(d(x,y)): 0 \leq d(x,y) \leq p\}<1$ for any $p>0$. Let $q$ be a fixed positive real number. Let us suppose that
$\sup\{\alpha(d(x,y)): 0 \leq d(x,y) \leq q\}=k$. Clearly, $k<1$. Hence, it follows from equation (\ref{eq 1 MT}) that if $x$ and $y$ be two elements in  $X$ with $d(x,y)<\min\{\varepsilon,q\}$, then
\begin{equation}
H(F(x),F(y))\leq \alpha(d(x,y))d(x,y) \leq kd(x,y).\label{eq 2 Th1}
\end{equation}
If $q \geq \varepsilon$, then $\min\{\varepsilon,q\}=\varepsilon$. Hence $d(x_{1},x_{2})<\min\{\varepsilon,q\}$. By (\ref{eq 2 Th1}), we have
$$H(F(x_{1}),F(x_{2}))\leq \alpha(d(x_{1},x_{2}))d(x_{1},x_{2}) \leq kd(x_{1},x_{2}).$$
Suppose that $q<\varepsilon$. Then $\min\{\varepsilon,q\}=q$. Since $X$ is metrically convex metric space, we can find $a,b \in \mathbb{R}$ and an isometry $\phi:[a,b]
\rightarrow X$ such that $\phi(a)=x_{1}$ and $\phi(b)=x_{2}$. For some $r$ with $0<r<q$, there exists a positive integer $m$ such that
 \begin{eqnarray*}
d(x_{1},x_{2})&=&d(\phi(a),\phi(b))=d(a,b)\\
&=& d(a,a+r)+d(a+r,a+2r)+\cdots+d(a+mr,b)\\
&=& d(x_{1},\phi(a+r))+d(\phi(a+r),\phi(a+2r))+\cdots+d(\phi(a+mr),x_{2}),
\end{eqnarray*}
where $m$ is such that $a+mr<b<a+(m+1)r$. Now,
$d(x_{1},\phi(a+r))=d(a,a+r)=r<\min\{\varepsilon,q\}$ and
hence we have from equation (\ref{eq 2 Th1}) that
$$ H(F(x_{1}),F(\phi(a+r))))\leq kd(x_{1},\phi(a+r))).$$
 Similarly, for any natural number $n<m$, we have $d(\phi(a+nr),\phi(a+(n+1)r))<\min\{\varepsilon,q\}$ and
hence by (\ref{eq 2 Th1}), we get
\begin{eqnarray*}
H(F(\phi(a+nr)),F(\phi(a+(n+1)r))))
 \leq kd(\phi(a+nr),\phi(a+(n+1)r))).
\end{eqnarray*}
Moreover, $d(\phi(a+mr),x_{2})<r<\min\{\varepsilon,q\}$ and thus by using (\ref{eq 2 Th1}) we have
$$H(F(\phi(a+mr)),F(x_{2}))\leq kd(\phi(a+mr),x_{2}).$$
Thus for any $x_{1},x_{2} \in X$ with $d(x_{1},x_{2})<\varepsilon$, we have
\begin{eqnarray*}
H(F(x_{1}),F(x_{2}))
&\leq & H(F(x_{1}),F(\phi(a+r)))+\cdots+H(F(\phi(a+mr)),F(x_{2}))\\
 &\leq & k[d(x_{1},\phi(a+r))+d(\phi(a+r),\phi(a+2r))+\cdots+d(\phi(a+mr),x_{2})]\\
 &=& kd(x_{1},x_{2}).
\end{eqnarray*}
\textbf{Case 2:} $P\neq \emptyset$.\\
Let $p_{0}=\inf P$. If  $p_{0}=0$, then we can find a sequence $\{p_{n}\}_{n}$ of $P$ such that
$p_{n} \rightarrow 0$ as $n \rightarrow \infty$. Since $p_{n} \in P$ for all $n$, $\sup\{\alpha(d(x,y)): 0 \leq d(x,y) \leq p_{n}\}=1$ for all $n \in \mathbb{N}$. By the properties of supremum, there exists a sequence $\{(x_{n},y_{n})\}_{n}$ such that $d(x_{n},y_{n}) \leq p_{n}$ and $\alpha(d(x_{n},y_{n}))>1-\frac{1}{n}$.
Hence $1-\frac{1}{n}\leq \alpha(d(x_{n},y_{n})) \leq 1$ for all $n \in \mathbb{N}$. Thus $d(x_{n},y_{n}) \rightarrow 0$ and $\alpha(d(x_{n},y_{n}))\rightarrow 1$
as $n \rightarrow \infty$. This is a contradiction as $\alpha:[0,\infty)\rightarrow [0,1)$ with $\limsup_{s \rightarrow t+}\alpha(s)<1$  $\forall t\in[0,\infty)$. Therefore $p_{0}\neq 0$. Thus for any $0<q_{0}<p_{0}$, $\sup\{\alpha(d(x,y)): 0 \leq d(x,y) \leq q_{0}\}=k_{0}<1$. Hence, it follows from equation (\ref{eq 1 MT}) that for any two elements $x$ and $y$ in  $X$ with $d(x,y)<\min\{\varepsilon,q_{0}\}$
\begin{equation}
H(F(x),F(y))\leq \alpha(d(x,y))d(x,y) \leq k_{0}d(x,y). \label{MT_Eq3}
\end{equation}
If $q_{0} \geq \varepsilon$, then $d(x_{1},x_{2})<\varepsilon=\min\{\varepsilon,q_{0}\}$. Hence from equation (\ref{MT_Eq3})
$$H(F(x_{1}),F(x_{2}))\leq \alpha(d(x_{1},x_{2}))d(x_{1},x_{2}) \leq k_{0}d(x_{1},x_{2}).$$
Let us assume that $q_{0}<\varepsilon$. Since $X$ is metrically convex metric space, we can find $a_{0},b_{0} \in \mathbb{R}$ and an isometry $\phi_{0}:[a_{0},b_{0}]
\rightarrow X$ such that $\phi_{0}(a_{0})=x_{1}$ and $\phi_{0}(b_{0})=x_{2}$. Similar to case 1, it follows that for some $r_{0}$ with $q_{0}>r_{0}>0$, there exists a positive integer $m_{0}$ such that
 \begin{eqnarray*}
d(x_{1},x_{2})= d(x_{1},\phi_{0}(a_{0}+r_{0}))+\cdots+d(\phi_{0}(a_{0}+m_{0}r_{0}),x_{2}),
\end{eqnarray*}
where $m_{0}$ is such that $a_{0}+m_{0}r_{0}<b_{0}<a_{0}+(m_{0}+1)r_{0}$. Moreover, for any non-negative integer $0\leq n_{0}<m_{0}$, we have similar to Case 1 that $d(\phi_{0}(a_{0}+n_{0}r_{0}),\phi_{0}(a_{0}+(n_{0}+1)r_{0}))<\min\{\varepsilon,q_{0}\}$ and hence by (\ref{MT_Eq3})
\begin{eqnarray*}
H(F(\phi_{0}(a_{0}+n_{0}r_{0})),F(\phi_{0}(a_{0}+(n_{0}+1)r_{0}))))
 \leq k_{0}d(\phi_{0}(a_{0}+n_{0}r_{0}),\phi_{0}(a_{0}+(n_{0}+1)r_{0}))).
\end{eqnarray*}
Thus for any $x_{1},x_{2} \in X$ with $d(x_{1},x_{2})<\varepsilon$, we have
\begin{eqnarray*}
&H\left(F(x_{1}),F(x_{2})\right) &\leq  H(F(x_{1}),F(\phi_{0}(a_{0}+r_{0})))+\cdots+H(F(\phi_{0}(a_{0}+m_{0}r_{0})),F(x_{2}))\\
 & &\leq  k_{0}\left[d(x_{1},\phi_{0}(a_{0}+r_{0}))+\cdots+d(\phi_{0}(a_{0}+m_{0}r_{0}),x_{2})\right]= k_{0}d(x_{1},x_{2}).
\end{eqnarray*}
Therefore $F$ is an $(\varepsilon,k)$-uniformly~locally~contractive~multi-valued mapping in both cases. Hence the proof is complete.
\end{proof}

\end{document}